\documentclass[11pt]{article}
\usepackage{amsthm,amssymb,amsmath}
\usepackage{epsfig}

\title{Chromatic number and complete graph substructures for degree sequences}

\author{Zden\v{e}k Dvo\v{r}\'ak\thanks{Supported in part through a postdoctoral
   position at Simon Fraser University and by the grant GA201/09/0197 of Czech Science Foundation.}~\thanks{On leave from: 
   Institute of Theoretical Informatics,
   Charles University, Prague, Czech Republic.}\\
  {Department of Mathematics}\\
  {Simon Fraser University}\\
  {Burnaby, B.C. V5A 1S6} \\
  email: {\tt rakdver@kam.mff.cuni.cz}
\and
  Bojan Mohar\thanks{Supported in part by an NSERC Discovery Grant (Canada),
  by the Canada Research Chair program, and by the
  Research Grant P1--0297 of ARRS (Slovenia).}~\thanks{On leave from:
  IMFM \& FMF, Department of Mathematics, University of Ljubljana, Ljubljana,
  Slovenia.}\\
  {Department of Mathematics}\\
  {Simon Fraser University}\\
  {Burnaby, B.C. V5A 1S6} \\
  email: {\tt mohar@sfu.ca}
}

\newtheorem{theorem}{Theorem}[section]
\newtheorem{lemma}[theorem]{Lemma}

\newtheorem{conjecture}[theorem]{Conjecture}

\newcommand{\DEF}[1]{{\em #1\/}}

\def\R {{\cal R}}

\begin{document}

\maketitle

\begin{abstract}
Given a graphic degree sequence $D$, let $\chi(D)$ (respectively $\omega(D)$,
$h(D)$, and $H(D)$) denote the maximum value of the chromatic number 
(respectively, the size of the largest clique, largest clique subdivision, and
largest clique minor) taken over all simple graphs whose degree sequence is $D$.
It is proved that $\chi(D)\le h(D)$. 
Moreover, it is shown that a subdivision of a clique of order $\chi(D)$ 
exists where each edge is subdivided at most once and the set of
all subdivided edges forms a collection of disjoint stars. This bound 
is an analogue of the Haj\'os Conjecture for degree sequences and,
in particular, settles a conjecture of Neil Robertson that degree sequences
satisfy the bound $\chi(D)\le H(D)$ (which is related to the Hadwiger Conjecture).
It is also proved that $\chi(D)\le \tfrac{6}{5}\,\omega(D)+\tfrac{3}{5}$ and 
that $\chi(D) \le \frac{4}{5}\,\omega(D) + \frac{1}{5}\Delta(D) + 1$,
where $\Delta(D)$ denotes the maximum degree in $D$. The latter inequality is
a strengthened version of a conjecture of Bruce Reed.
All derived inequalities are best possible.
\end{abstract}



\section{Introduction}

All graphs considered in this paper are simple, without loops or parallel edges.
A multiset of non-negative integers, usually written in the form of a non-increasing
sequence $d_1\ge d_2\ge \cdots\ge d_n$, is called a \DEF{graphic degree sequence}
if there exists a simple graph $G$ of order $n$ whose vertex degrees are 
$d_1,d_2,\dots,d_n$. Given a graph $G$, we denote by $D(G)$ its degree sequence,
and given a degree sequence $D$, we let
$$
   \R(D) = \{ G\mid D(G)=D \}
$$
denote the set of all \DEF{realizations} of $D$.
By $\delta(D)$ and $\Delta(D)$ we denote the minimum and the maximum degree
of $D$, respectively.

S.~B.~Rao introduced the following ordering for degree sequences: 
$D\preceq D'$ if there exist $G\in \R(D)$ and $G'\in\R(D')$ such that $G$ is 
an induced subgraph of $G'$. 
Motivated by the progress made by Robertson and Seymour \cite{RS20} on 
the well-quasi-ordering of graphs ordered by the graph minors relation,
he proposed the following conjecture.

\begin{conjecture}[S.~B. Rao, 1981]
\label{conj:Rao}
The degree sequences are well-quasi-ordered with respect to the relation $\preceq$.
\end{conjecture}

Recently, Chudnovsky and Seymour \cite{CS} announced a proof of this conjecture.

Let us introduce the following notation. If $D$ is a degree sequence, we let
$\chi(D)$ (respectively $\omega(D)$, $h(D)$, and $H(D)$) denote the maximum value 
of the chromatic number (respectively, the size of the largest clique, 
largest clique subdivision, and largest clique minor) taken over all graphs 
in $\R(D)$. Let us observe that $\omega(D)\le h(D)\le H(D)$.

Motivated by Rao's conjecture, Neil Robertson proposed a conjecture on degree
sequences that is a relaxation of the famous Hadwiger Conjecture claiming that
every graph with chromatic number $k$ contains a $k$-clique as a minor.
Despite many attempts, the Hadwiger Conjecture remains open, thus its relaxations 
are of high interest.

\begin{conjecture}[Robertson \cite{Ro}]
\label{conj:Robertson}
For every graphic degree sequence $D$, we have $\chi(D)\le H(D)$.
\end{conjecture}

In a recent work, Robertson and Song \cite{RS} proved a special case of
Conjecture~\ref{conj:Robertson} when the degree sequence contains at most two 
distinct degree values.

In this paper we consider a stronger version of Conjecture~\ref{conj:Robertson}
that is related to the Haj\'os Conjecture:

\begin{conjecture}
\label{conj:Hajos}
For every graphic degree sequence $D$, we have $\chi(D)\le h(D)$.
\end{conjecture}

This conjecture is particularly interesting, not only because it strengthens
Conjecture~\ref{conj:Robertson}, but also because the Haj\'os Conjecture for
graphs fails (Catlin \cite{Cat}; see also \cite{Th1,Th2}).

The main result of our work is a proof of Conjecture~\ref{conj:Hajos}
(see Theorems \ref{thm-hajos} and \ref{thm-hajos2}). 
It is shown that this conjecture holds in a quite strong
way. Namely, we prove that there is a graph $G\in \R(D)$ containing a subdivided
complete graph of order $\chi(D)$ such that each edge is subdivided at most once 
and the set of all subdivided edges forms a collection of disjoint stars. 
This, in particular, settles Conjectures~\ref{conj:Robertson} and \ref{conj:Hajos}.

We also address a question how close to $\chi(D)$ is the maximum clique
number $\omega(D)$. We prove (cf.\ Theorem~\ref{thm-sf}) that 
\begin{equation}
\label{eq:omega}
  \chi(D)\le\tfrac{6}{5}\,\omega(D)+\tfrac{3}{5}
\end{equation}
holds for every degree sequence $D$.
If $\chi(D)\le \tfrac{1}{2}n(D)$, where $n(D)$ is the
number of vertices for the degree sequence $D$, and if $\delta(D)\ge\chi(D)-1$,
then one can prove that $\omega(D)=\chi(D)$.
However, the situation changes when $\chi(D) > \tfrac{1}{2}n(D)$. For example,
any realization of the degree sequence $D=(5k-3)^{5k}$ (i.e., $d_i=5k-3$
for $i=1,\dots,5k$) is a complement 
of a union of cycles, thus it has $\chi(D)=3k$ (realized by
a join of $5$-cycles) and $\omega(D)=\lfloor\frac{5k}{2}\rfloor$ 
(realized by a complement of a $5k$-cycle).
If $k$ is odd, then $\chi(D)=\frac{6}{5}\omega(D)+\frac{3}{5}$.
This example shows that the inequality (\ref{eq:omega}) is best possible. 

Finally, we consider an analogue of Reed's Conjecture \cite{R}
(cf.\ Section~\ref{sect:3})
bounding the chromatic number by a convex combination of the clique number and
the maximum degree. Our Theorem \ref{thm-reed} shows that
$$
   \chi(D) \le \tfrac{4}{5}\,\omega(D) + \tfrac{1}{5}\Delta(D) + 1.
$$
This bound is best possible in the sense that equality holds for infinitely
many graphic sequences $D$ and that for every $\alpha > \tfrac{4}{5}$ and
every $\beta$, there exists a degree sequence $D$ such that
$$
   \chi(D) > \alpha\,\omega(D) + (1-\alpha)\Delta(D) + \beta.
$$

\section{Preliminary results}

Let us recall the following folklore results about graphic degree sequences,
cf., e.g.~\cite{Be}.

\begin{lemma}
\label{lemma-tree}
There exists a tree with degree sequence $d_1\ge d_2\ge\cdots \ge d_n$ if and only if\/ $d_n\ge 1$
and\/ $\sum_{i=1}^n d_i = 2n-2$.
\end{lemma}

\begin{lemma}
\label{lemma-bip}
Let $a_1\ge a_2\ge\cdots\ge a_n$ and $b_1\ge b_2\ge\cdots\ge b_m$ be two sequences of positive integers
such that $a_1\le m$, $\sum_{i=1}^n a_i=\sum_{i=1}^m b_i$ and $b_1\le b_m+1$.  Then there exists a bipartite
graph $G$ with parts $u_1,u_2,\ldots, u_n$ and $v_1,v_2,\ldots, v_m$ such that
the degree of $u_i$ is $a_i$ for $1\le i\le n$ and the degree of $v_i$ is $b_i$ for $1\le i\le m$.
\end{lemma}

Let us prove a slightly stronger statement:

\begin{lemma}\label{lemma-match}
Let $a_1\ge a_2\ge\cdots\ge a_n$ and $b_1\ge b_2\ge\cdots\ge b_m$ be two sequences of positive integers
such that $n\le m$, $a_1\le m$, $\sum_{i=1}^n a_i=\sum_{i=1}^m b_i$, and $b_1\le b_m+1$.  Let $A=\{u_1,u_2,\ldots, u_n\}$ and
$B=\{v_1,v_2,\ldots, v_m\}$ be two sets of vertices.  
Then there exists a bipartite graph $G$ with parts $A$ and $B$ such
that $G$ has a matching covering $A$, the degree of $u_i$ is $a_i$ for $1\le i\le n$ and the degree of $v_i$ is $b_i$ for $1\le i\le m$.
\end{lemma}

\begin{proof}
We prove the claim by induction.   The lemma holds if $n=1$, as in that case
$m\ge a_1=\sum_{i=1}^m b_i\ge m$, thus $a_1=m$ and $b_1=b_2=\cdots=b_m=1$, and
we take $G=K_{1,m}$.  Suppose now that $n\ge 2$ and that the claim is true for all
sequences $a'_1, \ldots, a'_{n'}$ and $b'_1, \ldots, b'_{m'}$ such that $n'<n$.
If $b_1=1$, then $G$ is a union of stars, thus assume that $b_1\ge 2$.
This implies that $a_1\ge 2$, as if $a_1=1$, then $m\ge n=\sum_{i=1}^n a_i=\sum_{i=1}^m b_i\ge m$
and $b_1$ would be equal to $1$.  On the other hand,
$mn\ge \sum_{i=1}^n a_i=\sum_{i=1}^m b_i\ge m(b_1-1)+1$, thus $b_1\le n$.

Consider the sequences $a_2-1, a_3-1, \ldots, a_{b_1}-1, a_{b_1+1}, \ldots,
a_n$ and $b_2-1, b_3-1, \ldots, b_{a_1}-1, b_{a_1+1}, \ldots, b_n$.  Let
$a'_1\ge a'_2\ge \cdots\ge a'_{n'}$ and $b'_1\ge b'_2\ge \cdots\ge b'_{m'}$ be
the positive elements of these two sequences.
If the first sequence is empty, then $a_2=a_3=\cdots=a_n=1$ and $n=b_1$, and 
hence $m+b_1-1\ge a_1+b_1-1=\sum_{i=1}^n a_i=\sum_{i=1}^m b_i\ge b_1+m-1$. 
Therefore equalities hold, and this is only possible when $a_1=m$ and
$b_2=b_3=\cdots=b_m=1$.  It follows that $n=b_1=2$ and $a_2=1$.  
In this case $G$ is $K_{1,m}$ with one
edge subdivided, which has a matching covering $A$ as $m\ge n=2$.

Therefore, we may assume that $n'\ge 1$.  Note that $\sum_{i=1}^{n'} a'_i=\sum_{i=1}^{m'} b'_i$,
which implies that $m'\ge 1$.  Also, observe that $b'_1\le b'_{m'}+1$.
Suppose first that $m'<m-1$, i.e., $b_{a_1}=1$, and thus $b_1=2$.
If $a_2>1$, then $n'=n-1$, and let $G$ be the graph obtained from the
union of stars $K_{1,a'_1}$, $K_{1,a'_2}$, \ldots, $K_{1,a'_{n'}}$ by adding the vertices
of $u_1$ and $v_1$ of degrees $a_1$ and $b_1$, respectively (with $u_1v_1$ in the matching),
joined to the appropriate vertices of the stars.  If $a_2=1$, then $G$ is a union of a matching
and the star $K_{1,a_1}$ with some (but not all) edges subdivided.
Therefore, we may assume from now on that $m'=m-1$.

Suppose now that $n'<n-1$, i.e., $a_{b_1}=1$.  In that case,
\begin{eqnarray*}
  (b_1-2)m-b_1+n+a_{b_1-1}+1 
    &\ge& (b_1-2)a_1+a_{b_1-1}+n-b_1+1 \\
    &\ge& \sum_{i=1}^n a_i=\sum_{i=1}^m b_i\ge (b_1-1)m+1.
\end{eqnarray*}
Thus $a_{b_1-1}\ge m-n+b_1\ge 2$, hence $n'=n-2$.  The sequences $a_1-1$, $a'_1$, $a'_2$, \ldots, $a'_{n'}$
and $b_2$, $b_3$, \ldots, $b_m$ satisfy the assumptions of the lemma; let $G'$ be the graph corresponding
to them.  We let $G$ be the graph obtained from $G'$ by adding the vertices $u_{b_1}$ and $v_1$
and joining $v_1$ with $u_1$, $u_2$, \ldots, $u_{b_1}$.  The edge $u_{b_1}v_1$ is added to the matching covering $A$.

Finally, consider the case when $n'=n-1$ and $m'=m-1$.
Note that $n'\le m'$.  If $b_1=n$, then $a'_1=a_2-1\le a_1-1\le m'$.
On the other hand, if $b_1<n$, then $(b_1+1)a_{b_1+1}\le\sum_{i=1}^n a_i=\sum_{i=1}^m b_i\le mb_1$, and thus
$a_{b_1+1}\le m-1$ and again, $a'_1\le m'$.  Therefore, the sequences $a'_1$, \ldots, $a'_{n'}$ and
$b'_1$, \ldots, $b'_{m'}$ satisfy the assumptions of the lemma; let $G'$ be the graph corresponding
to them.  We let $G$ be the graph obtained from $G'$ by adding the vertices $u_1$ and $v_1$
and joining $u_1$ with $v_1$, \ldots, $v_{a_1}$ and $v_1$ with $u_2$, \ldots, $u_{b_1}$.
The edge $u_1v_1$ is added to the matching covering $A$.
\end{proof}

We shall also need the following simple observation:

\begin{lemma}\label{lemma-12reg}
There exists a graph $G$ with $n\ge 1$ vertices, $e$ edges and\/
$1\le\delta(G)\le \Delta(G)\le 2$ if and only if
$e=n\ge 3$ or $e+1\le n\le 2e$ and $n\ge 2$.
\end{lemma}
\begin{proof}
Without loss of generality, if such a graph exists, then it is either a cycle, or a union of a (possibly empty) matching and a path
of length at least one.  The former is possible if and only if $e=n\ge 3$.  The latter is possible if and only if $e+1\le n\le 2e$ and $n\ge 2$.
\end{proof}

Rao~\cite{Rao2} proved the following:

\begin{theorem}\label{thm-rao}
A sequence $d_1\ge d_2\ge \cdots \ge d_n$ of nonnegative integers is 
a degree sequence of a graph $G$ with $\omega(G)\ge k$ if and only if\/
$\sum_{i=1}^n d_i$ is even, $d_k\ge k-1$, and
for\/ $0\le s\le k$ and\/ $0\le t\le n-k$,
\begin{eqnarray}
   \sum_{i=1}^s d_i &+& \sum_{i=k+1}^{k+t} d_i \,-\, 2{s+t\choose 2}\nonumber\\
   &\le& \sum_{i=s+1}^k \min(s+t, d_i+s-k+1) + \sum_{i=k+t+1}^n\min(s+t,d_i).
   \label{eq:1}
\end{eqnarray}
Furthermore, if these conditions are satisfied, then we can choose $G$ so that the vertices of
the clique of size $k$ have degrees $d_1,d_2,\ldots,d_k$.
\end{theorem}

Let us recall the characterization of graphic degree sequences by 
Erd\H{o}s and Gallai \cite{EG} in the following form:

\begin{theorem}\label{thm-eg}
A sequence $d_1, d_2, \ldots, d_n$ of non-negative integers
is graphic if and only if\/ $\sum_{i=1}^n d_i$ is even and for every 
$I\subseteq \{1,\ldots, n\}$,
\begin{equation}
  \sum_{i\in I}d_i \le 
    |I|(|I|-1)+\sum_{i\in \{1,\ldots, n\}\setminus I} \min(d_i, |I|).
\label{eq:2}
\end{equation}
\end{theorem}

If $d_1\ge d_2\ge \cdots \ge d_n$, then property (\ref{eq:2}) can be checked
only for subsets of the form $I=\{1,2,\dots,t\}$, $1\le t\le n$.

Yin and Li~\cite[Theorem~1.8]{YL} showed the following:

\begin{theorem}\label{thm-yinli}
Suppose that a graphic sequence $d_1\ge\cdots\ge d_n$ satisfies 
the following conditions:
$d_k\ge k-1$, 
$n\ge 2k$, and\/ $d_{2k}\ge k-2$.
Then it satisfies the assumptions of Theorem~\ref{thm-rao}.
\end{theorem}

We will need the following variation:

\begin{lemma}
\label{lemma-largecl}
A graphic sequence $d_1\ge\cdots\ge d_{2k-1}$ of length\/ $2k-1$ with
$d_{2k-1}\ge k-1$ 
satisfies the assumptions of Theorem~\ref{thm-rao} if and only if 
\begin{equation}
  \sum_{i=1}^{k-1}(d_i-d_k)+\sum_{i=k+1}^{2k-1}(d_k-d_i)\ge 2k-2-d_k.
  \label{eq:3}
\end{equation}
\end{lemma}

\begin{proof}
Consider the condition (\ref{eq:1}) of Theorem~\ref{thm-rao} with $s=0$ and $t=k-1$:
\begin{equation}
  \sum_{i=k+1}^{2k-1} d_i - 2{k-1\choose 2}\le \sum_{i=1}^k \min(k-1, d_i-k+1).
  \label{eq:4}
\end{equation}
As $d_i\le 2k - 2$, we have $\min(k-1, d_i-k+1)=d_i-k+1$ for $1\le i\le k$. 
Subtracting $kd_k$ from both sides of (\ref{eq:4}), we get
$$
  \sum_{i=k+1}^{2k-1} (d_i-d_k) - d_k - (k-1)(k-2) \le
  \sum_{i=1}^{k-1} (d_i-d_k) - k(k-1),
$$
which implies (\ref{eq:3}).

The above derivation actually shows that (\ref{eq:3}) and (\ref{eq:4}) are
equivalent. Therefore, we need to prove that (\ref{eq:3}) (or~(\ref{eq:4}))
implies the condition (\ref{eq:1}) of Theorem~\ref{thm-rao} for the choices
of $s$ and $t$ such that $0\le s\le k$, $0\le t\le k-1$ and 
either $s\neq 0$ or $t\neq k-1$.

If $d_k\ge k+t-1$ and $s<k$, then $\min(s+t, d_i+s-k+1)=s+t=\min(s+t, d_i)$ for $s+1\le i\le k$. Since the sequence $d_1,\ldots,d_n$ is graphic, (\ref{eq:1})
follows from (\ref{eq:2}) with $I=\{1, 2, \ldots, s, k+1, k+2, \ldots, k+t\}$.  
The same argument works if $s=k$ (independently of the value of $d_k$),
since the first sum on the right hand side of (\ref{eq:1}) vanishes.
Therefore, we may henceforth assume that $d_k\le k+t-2$ and $s\le k-1$.

As $d_1\le 2k-2$, $d_{2k-1}\ge k-1$, $d_i\ge d_k$ for $i\le k$ and $d_i\le d_k$ for $i\ge k$, it suffices to show that
\begin{eqnarray}
   s(2k-2) + td_k - 2{s+t\choose 2}
     &\le& (k-s)\min(s+t, d_k+s-k+1) + \nonumber\\
     && (k-1-t)\min(s+t,k-1).
   \label{eq:5}
\end{eqnarray}
As $d_k+s-k+1\le s+t$, the inequality (\ref{eq:5}) is equivalent to
\begin{eqnarray}
   (k+s)(k-1)+(s+t-k)d_k &\le& (s+t)(s+t-1)+s(k-s)+ \nonumber\\
   && (k-t-1)\min(s+t,k-1).
   \label{eq:6}
\end{eqnarray}

Let us first assume that $s+t\ge k$.
Since $d_k\le k+t-2$, it suffices to prove that 
$$
  (k+s)(k-1)+(s+t-k)(k+t-2)\le (s+t)(s+t-1)+s(k-s)+(k-t-1)(k-1),
$$
which is equivalent to $(k-t-2)(k-s-2)+k-3\ge 0$. 
If $t=k-1$, then $(k-t-2)(k-s-2)+k-3=s-1\ge 0$ (as $s+t\ge k$).
Similarly, if $s=k-1$, then $(k-t-2)(k-s-2)+k-3=t-1\ge 0$.
If $s\le k-2$ and $t\le k-2$, then $k\le s+t\le 2k-4$, hence $k\ge 4$ and $(k-t-2)(k-s-2)+k-3\ge 1$,
thus the condition is satisfied.

Let us now consider the remaining case when $s+t\le k-1$. 
As $\min(s+t,k-1)=s+t$, $d_k\ge k-1$, and the coefficient of $d_k$ on the
left hand side of (\ref{eq:6}) is negative, it suffices to prove that
$$
   (k+s)(k-1)\le(k-s-t)(k-1)+(s+t)(s+t-1)+s(k-s)+(k-t-1)(s+t).
$$
This is equivalent to $t\le st$.  If $s>0$, then this condition is satisfied. 
So, it remains to consider the case when $s=0$.
The condition (\ref{eq:1}) of Theorem~\ref{thm-rao} then becomes 
$$\sum_{i=k+1}^{k+t}d_i\le t(k-2) + \sum_{i=1}^k\min(d_i-k+1, t),$$
which is equivalent to
\begin{equation}
  t+\sum_{i=1}^t(d_{k+i}-k+1)\le \sum_{i=1}^k\min(d_i-k+1, t).
   \label{eq:7}
\end{equation}
Note that $d_{k+i}-k+1\le \min(d_{i+1}-k+1, t)$ since $d_{k+i}-k+1\le d_k-k+1\le t-1$ and $d_{k+i}\le d_{i+1}$, for $1\le i\le t$.  Therefore, (\ref{eq:7}) holds
if $d_1\ge k+t-1$.  Suppose that $d_1\le k+t-2$, i.e., $t\ge d_1-k+2$. Then 
$\min(d_i-k+1, t)=d_i-k+1$ for $1\le i\le k$.  Therefore, for $t\ge d_1-k+2$, the right-hand
side of (\ref{eq:7}) is independent of $t$, and the left-hand side is non-decreasing in $t$, and hence the condition is satisfied
for all $t$ if and only if it is satisfied for $t=k-1$, which is precisely
our original assumption (\ref{eq:4}).
\end{proof}

\section{Chromatic number and cliques}
\label{sect:3}

For graphs $G_1$, $G_2$, \ldots, $G_k$, let $G_1+G_2+\ldots+G_k$ be their {\em join},
i.e., the graph obtained from the disjoint union of $G_1$, $G_2$, \ldots, $G_k$ by adding 
all edges between $V(G_i)$ and $V(G_j)$ for all $i,j$ such that $1\le i<j\le k$.
A graph $G$ is {\em hypo-matchable} if for each $v\in V(G)$, $G-v$ has a perfect matching. A graph $G$ is
\DEF{$\chi$-critical} if for each $v\in V(G)$, $\chi(G-v)<\chi(G)$.

A graph $G$ is {\em basic} if 
\begin{itemize}
\item $\chi(G)\le\omega(D(G))$, or 
\item $G$ is $\chi$-critical, the number $n$ of vertices of $G$ is odd, $n=2m+1$, $\chi(G)=m+1$, $\omega(D(G))=m$ and the complement
of $G$ is hypo-matchable.
\end{itemize}

We say that $G$ is {\em nontrivial\/} if $\chi(G)>\omega(D(G))$.
Note that Lemma~\ref{lemma-largecl} describes the degree sequences of nontrivial basic graphs,
that is, the degree sequences of nontrivial basic graphs do not satisfy the condition (\ref{eq:3}).
The following lemma shows that when considering the behavior of $\chi(G)$ and $\omega(D(G))$, then we only care about the basic graphs.

\begin{lemma}
\label{lemma-decomp}
Any graph $G$ has an induced subgraph $G'$ such that $\chi(G)=\chi(G')$ 
and $G'$ is a join of basic graphs.
\end{lemma}

\begin{proof}
For a contradiction, assume that $G$ is a smallest counterexample.  
Let $n$ be the number of vertices of $G$.
As $G$ is not basic, $\chi(G)\ge \omega(D(G))+1$.
As $G$ is a smallest counterexample, $G$ is $\chi$-critical, and thus $\delta(G)\ge \chi(G)-1\ge\omega(D(G))$.
By Theorem~\ref{thm-yinli}, this implies that $n\le 2\omega(D(G))+1$.  Also, $G$ is not a join of two graphs,
i.e., the complement of $G$ is connected.

Consider a coloring $\varphi$ of $G$ by $\chi(G)$ colors, such that the set of vertices $B$ that belong to color classes
of size at least three is as small as possible.  Let $k=|V(G-B)|$ and let $c$ be the number of color classes of $\varphi$
restricted to $G-B$.  As $2\chi(G)\ge 2\omega(D(G))+2>n$, $2c>k$.
Let $M$ be the set of color classes of $\varphi$ of size two.  Note that
$M$ is a maximum matching in the complementary graph $\overline{G-B}$ and that
$|M|=k-c$. Conversely, any matching in $\overline{G-B}$ of size $k-c$
corresponds to a coloring of $G$ by $\chi(G)$ colors.  Also, any 
vertex $v$ of $G-B$ that is not incident with $M$ is adjacent
to all vertices of $B$, as otherwise if $v$ is not adjacent to a vertex $u\in B$, then we can set the color of $u$ to $\varphi(v)$,
thus decreasing the size of $B$.

By the Edmonds-Gallai theorem on maximum matchings in graphs,
there exists $T\subseteq V(G-B)$ and a matching $M'$ in
$\overline{G-B}$ such that each component of $\overline{G-B}-T$ is hypo-matchable, 
each edge of $M'$ is incident with exactly one vertex of $T$, and no component of
$\overline{G-B}-T$ is incident with more than one edge of $M'$. 
Moreover, each vertex in $T$ is incident with an edge in $M'$.
Let $C$ be the set of components of $\overline{G-B}-T$.
Let $t=|T|$ and let $h=|C|-t$ be the number of components of $\overline{G-B}-T$ that are not incident with an edge of $M'$.
Note that $h=2c-k>0$.  Consider the bipartite graph $H$ with parts $T$ and $C$, such that a vertex $u\in T$
and a component $K\in C$ are adjacent in $H$ if and only if there exists a vertex $v\in K$ such that $u$ and $v$ are non-adjacent in $G$.
Let $C'\subseteq C$ be the set of components that are covered by every matching in $H$ of size $t$.  By Hall's theorem,
there exists a set $T'\subseteq T$ such that $|T'|=|C'|$ and the vertices of $T'$ are adjacent in $G$ to all vertices of the components
in $C\setminus C'$.  Consider now a component $K\in C\setminus C'$ and a vertex $v\in K$.  There exists a matching in $H$ of size $t$ that does not
cover $K$, and this matching can be extended to a matching $M_1$ in $\overline{G-B}-T$ of size $k-c$ that does not cover $v$, as $K$ is hypo-matchable.
It follows that $v$ is adjacent in $G$ to all vertices of $B$.  
Since the choice of $K$ and $v$ was arbitrary, it follows that
all vertices of the components of $C\setminus C'$ are adjacent to all vertices of $B$.  Furthermore, note that for any
component $K\in C\setminus C'$, $\chi(G[V(K)])=(|V(K)|+1)/2$.  Also, $|C\setminus C'|\ge h>0$.

Let $G_0$ be the subgraph of $G$ induced by $\bigcup_{K\in C'} V(K)\cup B\cup T'$, 
and let $G_1$, $G_2$, \ldots, $G_a$ be
the subgraphs of $G$ induced by the vertex sets of the elements of $C\setminus C'$.  Let $G'$ be the join
of $G_0$, $G_1$, \ldots, $G_a$.  Observe that $\chi(G')=\chi(G)$ and that $G'$ is an induced subgraph of $G$.
As $G$ is $\chi$-critical, $G'=G$, and hence $T=T'$.  
However, the complement of $G$ is connected and $C\setminus C'\neq\emptyset$, 
thus $G_0$ must be an empty graph, i.e., $B=T=\emptyset$, and $a=|C|=1$.
It follows that $G=G_1$ is a nontrivial basic graph.
\end{proof}

As we have observed in the introduction, the degree sequence $D=(5k-3)^{5k}$ 
has $\chi(D)=3k$ (realized by the join of $5$-cycles) and $\omega(D)=\lfloor\frac{5k}{2}\rfloor$. 
Thus, $\chi(D) = \frac{6}{5}\omega(D)+\frac{3}{5}$ if $k$ is odd
(and $\chi(D) = \frac{6}{5}\omega(D)$ if $k$ is even).
Our next result shows that this example is the worst possible when
comparing $\chi(D)$ and $\omega(D)$.

\begin{theorem}\label{thm-sf}
For every graph $G$, $\chi(G)\le\frac{6}{5}\omega(D(G))+\frac{3}{5}$.
\end{theorem}

\begin{proof}
Suppose for a contradiction that $G$ is a smallest counterexample, and let $n$ be 
the number of vertices of $G$. By Lemma~\ref{lemma-decomp}, we may assume that 
$G$ is a join of basic graphs $G_1$, $G_2$, \ldots, $G_k$. Also, 
$$
  \chi(G) = \chi(G_1)+\chi(G_2+\cdots+G_k) 
        \le \chi(G_1) + \tfrac{6}{5}\omega(D(G_2+\cdots+G_k))+\tfrac{3}{5}
$$
and 
$$
  \omega(D(G))\ge \omega(D(G_1))+\omega(D(G_2+\cdots+G_k)).
$$
Thus $\chi(G_1)>\tfrac{6}{5}\omega(D(G_1))$,
and by symmetry, $\chi(G_i)>\tfrac{6}{5}\omega(D(G_i))$ and hence $G_i$ is nontrivial for $1\le i\le k$. Let $n_i=2m_i+1$ be the number
of vertices of $G_i$.  As $G_i$ is basic, $\chi(G_i)=m_i+1$ and $\omega(D(G_i))=m_i$.
Note that the smallest nontrivial basic graph $C_5$ has $n_i=5$, thus
$\chi(G)=\frac{n+k}{2}\le \frac{3n}{5}$.
On the other hand, $\delta(G)=\min\{n-n_i+\delta(G_i) : 1\le i\le k\} \ge \min\{n-m_i-1 : 1\le i\le k\}\ge \frac{n-1}{2}$,
and hence by Theorem~\ref{thm-yinli}, $\omega(D(G))\ge \frac{n-1}{2}$, and
$\chi(G)\le \frac{3n}{5} = \frac{6(n-1)/2+3}{5} \le \tfrac{6}{5}\omega(D(G))+\tfrac{3}{5}$.
\end{proof}

The chromatic number of any graph $G$ satisfies the following
trivial bounds:
$$
   \omega(G) \le \chi(G) \le \Delta(G)+1.
$$
Reed investigated general bounds on the chromatic number that can be
expressed as a convex combination of $\Delta(G)$ and $\omega(G)$.
He proposed the following

\begin{conjecture}[Reed \cite{R}]
\label{conj:Reed}
$\chi(G) \le \big\lceil\tfrac{1}{2}(\omega(G)+\Delta(G)+1)\big\rceil$
for every graph $G$.
\end{conjecture}

For degree sequences, we prove the following stronger bound of the same form:

\begin{theorem}\label{thm-reed}
For every graph $G$, 
~$\chi(G) \le \frac{4}{5}\,\omega(D(G)) + \frac{1}{5}\Delta(G) + 1$.
\end{theorem}

\begin{proof}
By considering a smallest counterexample, we see that $G$ is $\chi$-critical.
By Lemma \ref{lemma-decomp} we conclude that $G$ is a join of basic graphs $G_1,\ldots,G_k$.
Suppose first that $k=1$, i.e., $G$ is basic.  As $G$ is $\chi$-critical, $\Delta(G)\ge\delta(G)\ge\chi(G)-1$.
As $G$ is basic, $\omega(D(G))\ge \chi(G)-1$.  We conclude that
$\frac{4}{5}\,\omega(D(G)) + \frac{1}{5}\Delta(G) + 1\ge \chi(G)$.

Suppose now that $k>1$.  Let us first consider the case when one of the graphs $G_i$ in the join,
say $G_1$, is trivial, i.e., $\chi(G_1)\le \omega(D(G_1))$.  Observe that $\Delta(G_1)\ge \omega(D(G_1))-1$,
thus $\frac{4}{5}\,\omega(D(G_1)) + \frac{1}{5}\Delta(G_1)\ge \chi(G_1)-\frac{1}{5}$.
Note that $\Delta(G)\ge \Delta(G_1)+|V(G_2+\cdots+G_k)|\ge \Delta(G_1)+\Delta(G_2+\cdots+G_k)+1$.
It follows that $\chi(G)=\chi(G_1)+\chi(G_2+\cdots+G_k)\le
\left(\frac{4}{5}\,\omega(D(G_1)) + \frac{1}{5}\Delta(G_1)+\frac{1}{5}\right)+
\left(\frac{4}{5}\,\omega(D(G-G_1))+\frac{1}{5}\left(\Delta(G)-\Delta(G_1)-1\right)+1\right)
\le\frac{4}{5}\,\omega(D(G)) + \frac{1}{5}\Delta(G) + 1$.
This is in a contradiction with $G$ being a counterexample.  
It follows that each $G_i$ is a nontrivial basic graph.

Let $|V(G_i)| = 2m_i+1$, let $m=\min\{m_i\mid 1\le i\le k\}$,
and observe that $m\ge2$. We may assume that $m=m_1$.
Clearly, $\Delta=\Delta(G)\ge \Delta(G_1)+n-(2m+1) \ge n-m-1$, 
where $n=|V(G)|$. As in the proof of Theorem \ref{thm-sf} we conclude that
$\omega = \omega(D(G)) \ge \frac{n-1}{2}$.

>From the facts that $(m-2)(k-1)\ge 0$ and $n\ge (2m+1)k$, it follows that
$2m+5k-n+6\le 10$. Now,
\begin{eqnarray*}
  \chi(G) &=& \sum_{i=1}^k \chi(G_i) = \sum_{i=1}^k (m_i+1) 
            = \tfrac{1}{2}n + \tfrac{1}{2}k \\
    &=& \tfrac{4}{5}(\tfrac{n-1}{2}) + \tfrac{1}{5}(\tfrac{n-1}{2}) + \tfrac{k+1}{2}\\[1mm]
    &\le& \tfrac{4}{5}\omega + \tfrac{1}{5}(n-m-1) + \tfrac{1}{10}(2m+5k-n+6)\\[1mm]
    &\le& \tfrac{4}{5}\omega + \tfrac{1}{5}\Delta + 1.
\end{eqnarray*}
This is in a contradiction with $G$ being a counterexample, an the proof is complete.  
\end{proof}

Let us observe that the bound of Theorem \ref{thm-reed} is tight for 
the degree sequences $D=(5k-3)^{5k}$ ($k\ge 1$). The same examples also show that
the bound is best possible in the very strong sense as stated in the introduction. 

The proofs of Theorems \ref{thm-sf} and \ref{thm-reed} can be combined to
obtain the following stronger result. 

\begin{theorem}\label{thm-hajos2}
Let $D$ be a graphic degree sequence. There exists a realization
$G\in \R(D)$ with the following properties:
\begin{itemize}
\item[\rm (a)]
$\omega(G) \ge \tfrac{5}{6}\chi(D) - \tfrac{1}{2}$, and
\item[\rm (b)]
$\omega(G) \ge \tfrac{5}{4}\chi(D) - \tfrac{1}{4}\Delta(D) - \tfrac{5}{4}$.
\end{itemize}
\end{theorem}

\section{The Haj\'os Conjecture for degree sequences}

The Haj\'os variation, Conjecture \ref{conj:Hajos}, is true in a quite strong sense. 
Let $h_1(G)=r$ if $r$ is the largest integer such that $G$ contains a subgraph 
obtained from $K_r$ by first selecting vertex-disjoint subgraphs $S_1,\dots,S_a$ of $K_r$,
where each $S_i$ is isomorphic to a star $K_{1,n_i}$ ($1\le i\le a$),
and then subdividing each edge of these stars exactly once.

\begin{theorem}\label{thm-hajos}
For every graph $G$, $\chi(G)\le h_1(D(G))$.
\end{theorem}

\begin{proof}
Since $h_1(D(G))\ge \omega(D(G))$ and $h_1(G_1+G_2)\ge h_1(G_1)+h_1(G_2)$, Lemma~\ref{lemma-decomp} shows that we can
assume that $G$ is a nontrivial basic graph.  Let $n=2m+1$ be the number of vertices of $G$, and let
$d_1\ge d_2\ge \cdots\ge d_{2m+1}$ be the degree sequence of $G$.  Let 
$$
  \alpha=\sum_{i=1}^m (d_i-d_{m+1}) \quad\hbox{and}\quad 
  \beta=\sum_{i=m+2}^{2m+1} (d_{m+1}-d_i).
$$
As $\omega(D(G))<m+1$ and $\delta(G)\ge \chi(G)-1=m$, Lemma~\ref{lemma-largecl} implies that $\alpha + \beta < 2m-d_{m+1}$.

For $1\le i\le m$, let $a_i = d_i-(m-1)$.
Let $R=\frac{1}{2}(2m-d_{m+1}+\alpha+\beta)$.  Note that $R$ is an integer and $0\le \beta < R\le 2m-d_{m+1}-1\le m-1$. Since $\beta<R$, it follows that
$d_{2m-R}=d_{m+1}$.  Let $t_i=d_{m+1}-d_{2m-R+i}+1$ for $1\le i\le R$ and $t_{R+1}=d_{m+1}-d_{2m+1}+R-\beta$.
An easy calculation shows that $\sum_{i=1}^{R+1} t_i=2R$. All the numbers $t_i$ are positive, thus by Lemma~\ref{lemma-tree}, there
exists a tree $T_1$ with vertices $v_{2m-R+1}, v_{2m-R+2}, \ldots, v_{2m+1}$ such that 
the degree of $v_{2m-R+i}$ in $T_1$ is $t_i$ for $i=1,\dots,R+1$.

Our goal is to form a graph $G'$ with the degree sequence $D(G)$ on vertices $v_1, v_2, \ldots, v_{2m+1}$ in the following way. We
start with a graph $H_1$ consisting of a union of two cliques, one on the vertex set $A=\{v_1,v_2,\ldots,v_m\}$ and the other one
on the vertices $B=\{v_{m+1}, v_{m+2}, \ldots, v_{2m+1}\}$. 
Next, we shall delete the edges of $T_1$ or a slight modification of $T_1$. Finally, we shall add edges between $A$ and $B$ so that the degrees will be 
as requested. The modification of $T_1$ will be designed in such a way as to enable us 
to get the required subdivision of the $(m+1)$-clique, whose vertices of degree $m$
will be $v_1,\dots,v_{m+1}$, and the only subdivided edges will be some of the edges incident with $v_{m+1}$. 

Let us first consider the case when $d_{m+1}\ge m + \alpha$.  Then $m+R-\beta\le 2m-R$.
Let $T_2$ be a forest obtained from $T_1$ by choosing $R-\beta-1$ neighbors of $v_{2m+1}$, removing the edges joining them to $v_{2m+1}$,
and adding a matching between them and the vertices $v_{m+2}$, $v_{m+3}$, \ldots, $v_{m+R-\beta}$. Let $H_2=H_1-E(T_2)$. In order to get a graph $G'$ whose degree
sequence is $D(G)$, we have to add edges between $A$ and $B$. This has to be
done in such a way that each vertex $v_i\in A$ is incident with $a_i$ added edges
($1\le i\le m$), and each vertex $v_{m+i}\in B$ ($1\le i\le m+1$) is incident with
$b_i$ edges, where $b_i$ is as follows.
For $i=1$ and for $R-\beta+1\le i\le m-R$, we have $b_i=d_{m+1}-m$, and for
$2\le i\le R-\beta$ and for $m-R+1\le i\le m+1$, we have $b_i=d_{m+1}-m+1$.
Note that 
$$
  \sum_{i=1}^m a_i = \sum_{i=1}^m(d_i-m+1) = \alpha + m(d_{m+1}-m+1)
$$
and 
\begin{eqnarray*}
  \sum_{i=1}^{m+1} b_i &=& (d_{m+1}-m)(m+1)+R-\beta-1+R+1\\
  &=& \alpha + m(d_{m+1}-m+1) = \sum_{i=1}^m a_i.
\end{eqnarray*}
Thus we can apply Lemma~\ref{lemma-match}.
(This is obvious if $d_{m+1}>m$; if $d_{m+1}=m$, then $\alpha=0$ by the assumption
that $d_{m+1}\ge m + \alpha$; moreover, $\beta=0$ since $d_i\ge m$ for every $i$; 
thus $a_i=1$ for $i=1,\dots,m$ and $b_i=1$ for $i=2,\dots,m+1$, so
Lemma~\ref{lemma-match} can be applied trivially
for the sequence $b_2$, $b_3$, \ldots, $b_{m+1}$.) 
We conclude that there exists
a bipartite graph $H_3$ with parts $A=\{v_1,v_2,\ldots, v_m\}$ and $B=\{v_{m+1},v_{m+2},\ldots, v_{2m+1}\}$
such that the degree of $v_i$ is $a_i$ for $1\le i\le m$ and the degree of 
$v_{m+i}$ is $b_i$ for $1\le i\le m+1$, and $H_3$ has a matching $M$ covering $A$.

Let $G'=H_2\cup H_3$.
Observe that $D(G')=D(G)$ and $v_{m+1}$ is adjacent to all vertices 
in $B\setminus\{v_{m+1}\}$.
Together with the matching $M$, this gives a possibility to join $v_{m+1}$
with all vertices $v_1,\dots,v_m$ by using disjoint paths of length two.
This shows that $G'$ contains a subdivision of $K_{m+1}$, whose vertices of degree $m$ are $v_1,\dots,v_{m+1}$.
Each edge is subdivided at most once and all the subdivided edges are incident with $v_{m+1}$. Thus $h_1(G')\ge m+1$.

Suppose now that $d_{m+1}\le m+ \alpha - 1$. Note that $\alpha>0$ in this case.  Let $S$ be a set consisting of $v_{2m+1}$,
$R-\beta$ (arbitrarily chosen) neighbors of $v_{2m+1}$ in $T_1$ and of vertices $v_{m+2}$, $v_{m+3}$, \ldots, $v_{2m-R}$.
Note that $|S|=m-\beta$, $R-\beta < |S|$,
$2(R-\beta)=2m-d_{m+1}+\alpha-\beta \ge m-\beta+1 > |S|$, and $|S|\ge  m - (2m-d_{m+1}-\alpha-1) = d_{m+1}-m+\alpha+1\ge 2$.
Thus by Lemma~\ref{lemma-12reg} there exists a graph $T_3$ with vertex set $S$ and $R-\beta$ edges such that $1\le\delta(T_3)\le \Delta(T_3)\le 2$.
Let $S_i\subseteq S$ be the set of vertices of $T_3$ of degree $i$, for $i\in\{1,2\}$.  
Let $T_4$ be the graph obtained from $T_1$ by first removing the $R-\beta$ edges of $T_1$
induced by $S$ (i.e. those joining $v_{2m+1}$ and the $R-\beta$ chosen neighbors),
and then adding the $R-\beta$ edges of $T_3$. Let $H_4=H_1-E(T_4)$.
As in the previous case, we let $b'_i$ denote the number of edges that we need
to add between $A$ and the vertex $v_{m+1+i}$ ($i=1,\dots,m$) in order to form a graph whose degree sequences is $D(G)$. More precisely,  
let $b'_i=d_{m+1}-m+1$ for all $i$ such that $1\le i\le m$ and $v_{m+1+i}\not\in S_2$, and
$b'_i=d_{m+1}-m+2$ for all $i$ such that $v_{m+1+i}\in S_2$.  Let $a'_i=a_i-1$ for $1\le i\le d_{m+1}-m$
and $a'_i=a_i$ for $d_{m+1}-m+1\le i\le m$.  
If $a_i=1$ for some $i$, then we have $d_{m+1}=m$. This implies 
that $a'_i>0$ for $1\le i\le m$. 
By Lemma~\ref{lemma-match}, there exists a bipartite graph $H_5$ with parts
$A$ and $B\setminus \{v_{m+1}\}$ containing a perfect matching $M$, such that the degree of $v_i$ is $a'_i$
for $1\le i\le m$ and it is $b'_{i-m-1}$ for $m+2\le i\le 2m+1$.  

Let $G'$ be the graph obtained from $H_4\cup H_5$ by adding edges between 
$v_{m+1}$ and $v_1,v_2,\ldots,v_{d_{m+1}-m}$.
Again, we see that $D(G')=D(G)$ and $v_{m+1}$ is adjacent to all other vertices of $B$.
Together with the matching $M$, this yields a subdivision of $K_{m+1}$, where
each edge is subdivided at most once and all the subdivided edges are incident with $v_{m+1}$. We conclude that $h_1(G')\ge m+1$.
\end{proof}


\begin{thebibliography}{99}

\bibitem{Be} C. Berge, Graphs and Hypergraphs,
Second revised edition, North-Holland and Elsevier, 1976.

\bibitem{Cat} P. Catlin, 
Haj\'os' graph-coloring conjecture: variations and counterexamples, 
J. Combin. Theory Ser. B 26 (1979) 268--274.

\bibitem{CS} M. Chudnovsky, P. Seymour,
The proof of Rao's conjecture on degree sequences, 
talk at the Banff workshop on Graph Minors, 
BIRS, Banff, Alberta, September 28 -- October 3, 2008.

\bibitem{EG} P. Erd\H{o}s, T. Gallai,
Graphen mit Punkten vorgeschriebenen Grades,
Mat. Lapok. 11 (1960) 264--274.

\bibitem{Rao1} A.~R. Rao,
The clique number of a graph with given degree sequence, 
in: A.R. Rao (Ed.), Proceedings of the Symposium on Graph Theory, 
MacMillan and Co. India Ltd., I.S.I. Lecture Notes Series,
vol.~4, 1979, pp. 251--267.

\bibitem{Rao2} A.~R. Rao,
An Erd\H{o}s-Gallai type result on the clique number of a realization of a degree sequence,
unpublished.

\bibitem{R} B. Reed, 
$\omega$, $\Delta$, and $\chi$, J. Graph Theory 27 (1998) 177--212.

\bibitem{Ro} N. Robertson,
On Rao's Conjecture, 
a talk at the workshop Graph Theory, Oberwolfach, January 2005.

\bibitem{RS20} N.~Robertson and P.~D.~Seymour,
Graph minors. XX. Wagner's Conjecture, 
J.~Combin.\ Theory Ser.~B 92 (2004) 325--357.

\bibitem{RS} N. Robertson, Z. Song,
Hadwiger number and chromatic number for near regular degree sequences,
J. Graph Theory, in press.

\bibitem{Th1} C. Thomassen, 
Some remarks on Haj\'os' conjecture,
J. Combin. Theory Ser. B 93 (2005) 95--105. 

\bibitem{Th2} C. Thomassen, 
Haj\'os' conjecture for line graphs,
J. Combin. Theory Ser. B 97 (2007) 156--157. 

\bibitem{YL} J.-H. Yin, J.-S. Li, 
Two sufficient conditions for a graphic sequence to have a realization with prescribed clique size,
Discrete Math. 301 (2005) 218--227.

\end{thebibliography}
\end{document}